\documentclass[11pt,a4paper,leqno,twoside]{amsart}
\usepackage{amsmath,amssymb,amsthm,enumerate}
\usepackage[colorlinks=true,urlcolor=blue, citecolor=red,linkcolor=blue,linktocpage,pdfpagelabels,bookmarksnumbered,bookmarksopen]{hyperref}
\numberwithin{equation}{section}
\newif\ifcomment \commentfalse
\def\commentON{\commenttrue}

\long\outer\def\BC#1\EC{\ifcomment \sloppy \par \# \ldots\dotfill
{\em #1} \dotfill \# \par \fi } \commentON

\newcommand{\remove}[1]{}
\newtheorem{theorem}{Theorem}[section]

\newtheorem{proposition}[theorem]{Proposition}
\newtheorem{lemma}[theorem]{Lemma}

\newtheorem{remark}[theorem]{Remark}

\newcommand{\nor}{\Arrowvert}
\newcommand{\utp}{\tilde{u}_p}
\def\tv{\tilde v}
\def\up{u_p}

\def\tu{\widetilde{u}}
\def\e{{\varepsilon}}

\def\na{\nabla}
\def\d{\delta}
\def\g{{\gamma}}
\def\G{{\Gamma}}
\def\D{\Delta}
\def\L{{\Lambda}}
\def\l{{\lambda}}
\def\a{{\alpha}}
\def\b{{\beta}}

\def\de{\partial}

\def\cS{\mathcal{S}}
\def\cC{\mathcal{C}}
\newcommand{\R}{\mathbb{R}}
\newcommand{\N}{\mathbb{N}}
\newcommand{\pa}{\partial}
\newcommand{\B}{B}
\begin{document}
\title{Bifurcation and symmetry breaking for the Henon equation}
\author[A.~L.~Amadori, F.~Gladiali]{Anna Lisa Amadori$^\dag$, Francesca Gladiali$^\ddag$}
\date{\today}
\address{$\dag$ Dipartimento di Scienze Applicate, Universit\`a di Napoli ``Parthenope", Centro Direzionale di Napoli, Isola C4, 80143 Napoli}
\address{$\ddag$ Matematica e Fisica, Universit\`a di Sassari,via Piandanna 4, 07100 Sassari}
\email{annalisa.amadori@uniparthenope.it  fgladiali@uniss.it}
\begin{abstract}
\noindent
In this paper we consider the problem
$$
\left\{
\begin{array}{ll}
-\Delta u=|x|^{\a}u^p &\text{ in }\B ,\\
u>0& \hbox{ in }\B ,\\
u=0 & \hbox{ on }\pa \B ,
\end{array}\right.
$$
where $\B $ is the unit ball of $\R^N$, $N\ge 3$, $p>1$ and $0<\a\leq 1$. We prove the existence
of (at least) one branch of nonradial solutions that bifurcate from the radial ones and that this branch is  unbounded.


\end{abstract}

\maketitle

\section{Introduction}\label{s1}
In this paper we consider the problem
\begin{equation}\label{1}
\left\{
\begin{array}{ll}
-\Delta u=|x|^{\a}u^p & \hbox{ in }\B ,\\
u>0& \hbox{ in }\B ,\\
u=0 & \hbox{ on }\pa \B ,
\end{array}\right.
\end{equation}
 introduced by Henon in 1973 in the study of stellar cluster in spherically symmetric setting \cite{H}.
Here $\B $ is the unit ball in $\R^N$ with $N\geq 3$, $\a>0$, and $1<p<p_{\a}$, with $p_{\a}=\frac{N+2+2\a}{N-2}$.

In the subcritical case, i.e.~for $1<p<\frac{N+2}{N-2}=2^*-1$, where $2^*=\frac{2N}{N-2}$ is the usual critical Sobolev exponent, standard embedding arguments yield that problem \eqref{1} has at least one solution. The critical and supercritical case, i.e.~$p\geq \frac{N+2}{N-2}$, are not so simple instead, because the presence of the term $|x|^{\a}$ brings to a new critical exponent $p_{\a}$: an {\textit{ad-hoc}} Pohozaev identity implies nonexistence of solutions for $p\geq p_{\a}$. % (see the Appendix for details).
Actually $p_{\a}$ is the analogous of the critical Sobolev exponent for the Henon problem, and separates the threshold between existence and nonexistence of positive solutions.
The following fundamental existence result is due to Ni.
\begin{theorem}[Ni, \cite{Ni}]\label{tNi}
The elliptic boundary value problem \eqref{1} possesses a positive radial solution provided $p\in\left(1, p_{\a}\right)$.
\end{theorem}
This result is achieved by the application of the Mountain Pass Lemma in a space of radial
functions, and the solution obtained is therefore radial. It is also unique among radial functions, see \cite{NN}. In all the following we shall denote by $u_p$ this radial solution.

Besides, since the term $|x|^{\a}$ is radially increasing, the moving plane method of Gidas, Ni, Nirenberg cannot be applied: non radial solutions exist and the symmetry breaking occurs.
In the literature there are some results on the existence of non-radial solutions of \eqref{1}, see for example \cite{BS}
and \cite{S} where the authors find nonradial solutions minimizing the functional associated to \eqref{1} in some suitable symmetric spaces, and \cite{PS} and \cite{P} where nonradial solutions are constructed by the well known Liapunov-Schmidt finite dimensional reduction method, for $p=2^*-1-\epsilon$. These last solutions concentrate at the boundary $\de B$ as $\epsilon\to 0$.\\
% \cite{SSW} and references therein.
%The literature about nonradial solutions is quite rich, see for example \cite{SSW} and references therein.
Smets, Su and Willem in \cite{SSW} 
investigated minimal energy solutions, in the subcritical case. They proved that the ground state solution is nonradial, provided that $\a$ is above a critical value $\a^*(p)$. Moreover $\a^*(p)$ goes to  zero as $p\to \frac{N+2}{N-2}$ and $\a^*(p)$ goes to $+\infty$ as $p\to 1$, showing that, for $\a$ fixed, the ground state solution is radial. In \cite{CP} the authors studied the asymptotic behavior of the ground state solution as $p\to 2^*-1$. Again these solutions do concentrate on the boundary of $\B$, (at a single point).
Hence problem \eqref{1} has at least two solutions, for $\a$ large enough fixed and $p$ in a compact subset of $(1,\frac{N+2}{N-2})$.
For some existence results in more general domains see \cite{cowan13supercritical} and \cite{GG}.
%Other existence results  are \cite{BS}, \cite{cowan13supercritical}, \cite{GG} and \cite{S}.

In this paper we want to find %study the bifurcation of 
nonradial solutions of \eqref{1} studying the bifurcation from the radial solution, when the exponent $p$ varies in the range of existence $(1,p_\a)$, for fixed $\a$, and obtain the following result:
\begin{theorem}\label{tbif}
Let $\alpha\in (0,1]$ be fixed. Then there exists at least one exponent $\bar p\in (1,p_{\a})$ such that a nonradial bifurcation occurs at $(\bar p,u_{\bar p})$. The bifurcating solutions are positive and form a continuum which is unbounded in the Holder space $C^{1,\g}_0(\overline \B )$.
\end{theorem}
Unfortunately we don't know if such bifurcation occurs at $\bar  p$ greater or less than $2^*$, and if the branch lives for $p$ above the critical Sobolev exponent. %We do not even know if our branch is made up of the least energy solutions of \cite{SSW}.
For sure we can say that the nonradial solutions we find in Theorem \ref{tbif} do not coincide with those found by \cite{BS, P, PS, S} for reasons of symmetry. Indeed our solutions inherit some of the symmetries of the domain. Then, or they coincide with the ground state solutions of \cite{SSW}, or they give rise to other new solutions.
To our knowledge, the breaking of symmetry given by bifurcation from the radial solution is observed here for the first time, in the framework of the Henon problem.
An analogous effect has been found about  the problem $-\Delta u=u^p$ in an annulus in \cite{GGPS, G, GP}. The authors proved, among other results, the bifurcation of infinitely many global branches from the unique positive radial solution.
The basic idea is that a change in the Morse index of the radial solution causes a change in the Leray-Schauder degree of an associated map, and then bifurcation occurs.
Their techniques partially apply to problem \eqref{1} because  the term $|x|^{\a}$ acts as the presence of a hole in $\B $, as previously observed by Serra in \cite{S}.
The difference here is that the Morse index remains bounded.

In Section \ref{s2} we study the linearized operator at the radial solution,  and characterize the degeneracy points - which are the candidates for the bifurcation - and the Morse index of $u_p$ by means of the first eigenvalue of a suitable Sturm-Liouville problem (see Theorem \ref{t1}).
The change in the Morse index is  a byproduct of the asymptotic behavior of the radial solution $u_p$ as $p\to 1$ and as $p\to p_{\a}$, analyzed in Section \ref{s3}.
We can  prove that it goes from $1$ to $N+1$, as $p$ goes from $1$ to $p_\a$, provided that $0<\alpha\le 1$ (see Theorem \ref{uniq} and Proposition \ref{uniq2}).  
Even though we do not believe that the assumption $0<\alpha\le 1$ is sharp, there is evidence that some upper bound on the value of $\alpha$ is necessary.
Actually the paper \cite{GGN13nonradial} addresses to a similar problem in the whole space $\R^n$, for $p=p_\a$; in that case the Morse index is nondecreasing w.r.t.~$\alpha$, and changes as $\alpha$ crosses the even integers. %{\textcolor{blue}{Francesca controlla e correggi!}}
Some of the results stated in Section \ref{s3} are interesting by themselves:  in Theorem \ref{uniq} we prove uniqueness of the solution of \eqref{1} for $p$ near $1$. This uniqueness result provides an alternative proof of the asymptotic behavior of $\alpha^*(p)$ (see \cite{SSW}) for $p \to 1$. %, improving the uniqueness result for ground state solutions of \cite{SSW} that provides %another proof of the asymptotic behavior of $\a^*(p)$ as $p\to 1$. % established by \cite{SSW}.
 Finally Theorem \ref{tbif} is proved in Section \ref{s4}. In the Appendix we give the details of some known facts that we use in Section \ref{s2}.

\section{Preliminaries on the radial solutions}\label{s2}
In this section we will prove some results on $\up$, the radial solution of the problem \eqref{1}, by studying the related linearized operator.
% and in the computation of its Morse index, i.e.~the number of negative eigenvalues of the linearized operator, counted with their multiplicity. We begin by introducing some notations. In all the following, we denote by $\up$ the unique radial solution to \eqref{1}, as  $p\in(1,p_{\a})$.
In particular we address to degeneracy and characterize that exponents $p$ such that the linearized problem
\begin{equation}
\label{L}\left\{
\begin{array}{ll}
-\Delta v = p|x|^{\a}u_p^{p-1}v & \hbox{ in } \B ,\\
v=0 & \hbox{ on }\pa \B ,
\end{array}\right.
\end{equation}
has nontrivial solutions. We also compute  the Morse index of the solution $u_p$, i.e.~ is the number of negative eigenvalues of the standard eigenvalue problem linked to \eqref{L}, each counted with its multiplicity.
To these purposes it is convenient to consider a slightly different  eigenvalue problem:
\begin{equation}\label{A}
\left\{\begin{array}{ll}
-\Delta v=\L p |x|^{\a}\up^{p-1} v  & \hbox{ in }\B \\
v=0 & \hbox{ on }\de \B 
\end{array}\right.
\end{equation}
where $\L$ is a real number. It is obvious that problem \eqref{L} has a nontrivial solution if and only if  problem \eqref{A} admits $\L=1$ as an eigenvalue.
Besides the Morse index of  $\up$ coincides with the number of eigenvalues of \eqref{A} less than $1$ (counted with their multiplicity). This straightforward  relation is explained in details in Lemma \ref{equiv-Morse-index} of the Appendix. 
By taking advantage of the radial symmetry, one can deal with a family of one-dimensional problems.%, instead of \eqref{A}.
\begin{remark}\label{passaggioa1dim}
  Let $Y_{kj}(\theta)$ be the spherical harmonic functions, i.e.~the solution to
\[%\begin{equation}\label{arm-sfer}
-\Delta_{S^{N-1}}Y_{kj}(\theta)=\mu_k Y_{kj}(\theta) \quad \text{ for }j=1,\dots,m(k).
\]%\end{equation}
Here $-\Delta_{S^{N-1}}$ is the Laplace-Beltrami operator on the $(N-1)$-dimensional sphere $S^{N-1}$, $\mu_k=\frac{N+2k-2}{N+k-2}$ is its sequence of eigenvalues, and $m(k)=\binom{k+N-2}{k}$ is the dimension of the relative eigenspace.
We decompose solutions to \eqref{A} along the spherical harmonic functions and write
\[%\begin{equation}\label{harmproj}
v(r,\theta)=\sum_{k,j} \psi_{kj}(r)\,Y_{kj}(\theta), \]%\end{equation}
where $ \psi_{kj}$ is the projection of $v$ along $Y_{kj}$. %$\disp  \psi_{kj}(r)=\int_{S^{N-1}} v(r,\theta)\, Y_{kj}(\theta)\, d\theta$. \\
Inserting this formula in \eqref{A}, one realizes that the eigenvalues problem \eqref{A} is in correspondence with the family of one-dimensional  eigenvalues problems
\begin{align}\label{Arad}
\left\{\begin{array}{lr}
 -(r^{N-1}\psi_k')'+\mu_k r^{N-3}\psi_k(r)=\L p r^{N-1+\a}\up^{p-1} \psi_k(r)   &  \quad \mbox{ for all } r\in (0,1),   \\
 \psi_k(0)=0\; ,\;\psi_k(1)=0 & \mbox{ if } k>0, \\
\psi_0'(0)=0\; ,\; \psi_0(1)=0, & \mbox{ if } k=0.
\end{array}\right.\end{align}
Indeed $\L$ is an eigenvalue for \eqref{A} if, and only if, there exists at least one $k$ such that $\L$ is an eigenvalue for \eqref{Arad}.
Moreover the dimension of the relative eigenspace is obtained by summing the multiplicity $m(k)$, on all $k$ so that $\L$ is an eigenvalue for \eqref{Arad}.
\end{remark}
 For every fixed $p\in(1,p_\a)$ and $k\in\N$, the eigenvalue problem \eqref{Arad} is of Sturm-Liouville type, and therefore it has a sequence of simple eigenvalues $\L_{i,k}(p)$, $i\in \N$.
As regards both degeneracy and Morse index, only the first eigenvalue related to the first radial mode $\mu_1=N-1$ plays a role.
\remove{, i.e.~the first eigenvalue of
We show that only the first eigenvalue related to the first radial mode $\mu_1=N-1$ plays a role, i.e.~the first eigenvalue of
\begin{equation}\label{Arad1}
\left\{\begin{array}{ll}
-(r^{N-1}\psi_k')'+(N-1) r^{N-3}\psi (r)=\L p r^{N-1+\a}\up^{p-1} \psi (r)  , \quad & \mbox{ for all } r\in (0,1)\\
\psi(0)=0\,,\,\psi_k(1)=0. &
\end{array}\right.
\end{equation}}
This fact, which heavily depends on the assumption $\alpha\le 1$,  will be crucial to  depict the asymptotic behavior of the Morse index as $p\to 1$ and $p\to p_\a$ in next Section. 
%This allows to  characterize degenerates radial solutions to \eqref{1}.
%
\begin{theorem}\label{t1}
Let $1<p<p_{\a}$  and  $0<\a\leq 1$. Then $\up$ is degenerate if and only if $\L_{1,1}(p)=1$.
Moreover its Morse index  can take only two values: it is equal to $1$ if $\L_{1,1}(p)\ge  1$, or equal to $N+1$ if $\L_{1,1}(p)< 1$.
\end{theorem}

\begin{proof}
The proof %of Theorem \ref{t1} 
is split in several steps. With Remark \ref{passaggioa1dim} in mind, we analyze separately any radial mode $\mu_k$, and  study the related one dimensional eigenvalue problem \eqref{Arad}. \\

\paragraph{Step 1 - $k=0$.}
As $\mu_0=0$, investigating the first radial mode  \eqref{Arad} means looking for radial solutions to \eqref{L}.
The function $\up$ is an eigenfunction corresponding to the first eigenvalue $\L_{1,0}=1/p<1$. We show that  $\L_{i,0}>1$ for all $i\ge 2$.
In doing so, we get  that  \eqref{L} does not have nontrivial radial solutions, i.e.~$u_p$ is radially nondegenerate. This result deserves to be stated separately, because it is of some interest by itself.
\begin{proposition}\label{lin-rad}
The linearized problem \eqref{L} does not admit any nontrivial radial solution.
\end{proposition}
The proof of Proposition \ref{lin-rad} requests some preliminary knowledge about the radial solution $u_p$, whose proof is postponed  into the Appendix for  reader's comprehension.
\begin{lemma}\label{der-seconda}
Let $\up(r)$, $0\le r\le 1$, be the unique radial solution of \eqref{1}. Then $\up'<0$ for $0<r\le 1$, moreover
\begin{equation}\label{R4}
\int_0^1r^{N-1}(v')^2 \, dr-p\int_0^1 r^{N-1+\a}\up^{p-1}v^2\, dr-(p-1)\frac{\left(\int_0^1 r^{N-1+\a}\up^{p}v\, dr\right)^2}{\int_0^1 r^{N-1+\a}\up^{p+1}\, dr}\geq 0
\end{equation}
for any radial function $v$ in $H^1_0(\B )$.
\end{lemma}

\begin{proof}[Proof of Proposition \ref{lin-rad}]
Arguing by contradiction, let us assume that there exists a nontrivial radial solution $\overline v$ of (\ref{L}). Then $\overline v$
is an eigenfunction of \eqref{Arad} corresponding to $\mu_k=0$ and $\L=1$.
Because $\L_{1,0}=1/p<1$, there should be some $i\ge 2$ so that $\L_{i,0}=1$.\\
The second eigenfunction $\psi_{2,0}$ satisfies
\begin{equation}\label{Arad02}
\left\{\begin{array}{ll}
-(r^{N-1}\psi_{2,0}')'=\L_{2,0} pr^{N-1+\a}\up^{p-1} \psi_{2,0}(r)  & \hbox{ in }(0,1)\\
\psi_{2,0}'(0)=0\,,\,\psi_{2,0}(1)=0.
\end{array}\right.
\end{equation}
Moreover it is orthogonal in $H^1_0(\B )$ to $\psi_{1,0}=\up$ so that, in radial coordinates, we have $\int_0^1 r^{N-1}\up'\psi_{2,0}'\, dr=0$. Using \eqref{Arad02} and integrating by parts, we get
\begin{align*}
&\int_0^1 r^{N-1}\left( \psi_{2,0}'\right)^2\, dr=\L_{2,0}\, p\int_0^1 r^{N-1+\a}\up^{p-1}\psi_{2,0}^2\, dr\\%\label{S1}\\
&\int_0^1 r^{N-1+\a}\up^{p}\psi_{2,0}\, dr=0\label{S2}.
\end{align*}
Taking $v=\psi_{2,0}$ in \eqref{R4}, and inserting these two inequalities gives
\[ \left(\L_{2,0}-1\right)p\int_0^1 r^{N-1+\a}\up^{p-1}\psi_{2,0}^2\, dr\geq 0. \]
Hence $\L_{2,0}\geq 1$ and therefore $\L_{2,0}=1$.
 Further, as $\psi_{2,0}$ is the second eigenfunction,  it has two nodal regions. Let us say that $\psi_{2,0}$ has constant sign on  two intervals $(0,r_0)$ and $(r_0,1)$, with $r_0\in(0,1)$, $\psi_{2,0}(r_0)=0$. This implies in turn that the first eigenvalue of the linearized problem \eqref{A} in the smaller ball $B_{r_0}(0)$ is equal to $1$.
%and the same holds in the annulus $\B \setminus  B_{r_0}(0)$.\\
Next, we consider the function $z_p:=ru'_p +\frac2{p-1}u_p$. It satisfies
\begin{equation}\label{zp}
\left\{
\begin{array}{ll}
-(r^{N-1} z_p')'=pr^{N-1+\a}u_p^{p-1}z_p & \hbox{ in }(0,1),\\
z_p(0)>0,\quad
z_p(1)<0 .&\
\end{array}\right.
\end{equation}
Moreover $z_p'(0)=\frac{p+1}{p-1}\up'(0)=0$ so that $z_p$ is a radial function in $H^1(\B )$.
From \eqref{zp} we know that $z_p$ changes sign on $(0,1)$ at least once: let $d\in(0,1)$ be such that $z_p(d)=0$ and $z_p>0$ on $(0,d)$.
The function $z_p$ is an eigenfunction of the linearized problem  \eqref{A} related to the first eigenvalue $1$, in the ball $B_d(0)$. By the strict monotonicity of the first eigenvalue with respect to the inclusion of domains, it follows that $d=r_0$ and therefore $\psi_{2,0}=Cz_p$ for some constant $C\neq 0$. This is not possible since $z_p$ does not satisfy the boundary condition in $r=1$ and proves the Lemma.
\end{proof}

\paragraph{Step 2 - $k=2$.}
For any $p$ and $k=2$, all eigenvalues of problem \eqref{Arad} are greater than 1. It suffices to prove that first eigenvalue $\Lambda_{1,2}(p)>1$ for every $p\in(1,p_{\alpha})$.
To this aim, we introduce the function $w_p:=-\up'$. In radial coordinates, $w_p$ solves
\begin{equation}\label{wp}
\left\{\begin{array}{ll}
-(r^{N-1}w_p')'=\left(-(N-1)r^{N-3}+pr^{N-1+\a}\up^{p-1}-\a r^{N-2+\a}\frac{\up^p}{w_p}\right)w_p  & \hbox{ in }(0,1)\\
w_p(0)=0\,,\quad  w_p(r)>0 \; \hbox{ as } r\in(0,1).
\end{array}\right.
\end{equation}
Next, let  $\psi_{1,2}$  a first positive eigenfunction of \eqref{Arad} corresponding to $k=2$. Multiplying \eqref{Arad}  for $w_p$ and integrating over $(0,1)$ we get
\begin{align*}
-\int_0^1 (r^{N-1}\psi_{1,2}')' \, w_p\, dr  = \int_0^1 r^{N-3}\psi_{1,2}w_p\left(-2N+\L_{1,2} pr^{2+\a}\up^{p-1}\right)\, dr \\
\intertext{and integrating by parts this yields} %\begin{equation}\label{a}
\int_0^1 r^{N-1}\psi_{1,2}' \, w_p'\, dr +\psi_{1,2}'(1)(-w_p(1))=\int_0^1 r^{N-3}\psi_{1,2}w_p\left(-2N+\L_{1,2} pr^{2+\a}\up^{p-1}\right)\, dr.
\end{align*}
Besides, multiplying \eqref{wp} for $\psi_{1,2}$ and  integrating by parts  over $(0,1)$  we get
\[%\begin{equation}\label{b}
\int_0^1 r^{N-1}w_p'\psi_{1,2}'\, dr=\int_0^1 r^{N-3}w_p\psi_{1,2}\left(-(N-1)+ pr^{2+\a}\up^{p-1}-\a r^{1+\a}\frac{\up^p}{w_p} \right)\, dr.
\]%\end{equation}
Subtracting the two obtained equalities yields
\[\psi_{1,2}'(1)(-w_p(1))=\int_0^1 r^{N-3}\psi_{1,2}w_p\left[(\L_{1,2}-1)pr^{2+\a}\up^{p-1}-N-1+\a(N+\a)g(r)\right]\, dr\]
where $\displaystyle  g(r)={r^{1+\a}\up^p}/{(N+\a)w_p}$.
Because  $\psi_{1,2}'(1)<0$ and $w_p(1)>0$ by  Hopf boundary Lemma, it follows that
\[%\begin{equation}\label{c}
(\L_{1,2}-1)\, p\int_0^1 r^{N-1+\a}\up^{p-1}\psi_{1,2}w_p\, dr>\int_0^1 r^{N-3}\psi_{1,2}w_p\left[N+1-\a(N+\a)g(r)\right]\, dr,
\]%\end{equation}
and the proof is completed after checking that
%\begin{equation}\label{gr}
$0 <g(r)<1$ %\quad \mbox{ for any } r\in(0,1).\]
for any $r\in(0,1)$.
The first inequality holds because  $w_p=-\up'>0$ by Lemma \ref{der-seconda}.
Concerning the second one, we have that $g(1)=0$ and
\[\lim_{r\to 0^+}g(r)=\lim_{r\to 0^+}\frac{\up^p(r)}{(N+\a)} \; \frac{r^{N+\a}}{-r^{N-1}u_p'(r)}=1.\]
Indeed $-(r^{N-1}\up')'=r^{N-1+\a}u_p^{p}$ by equation  \eqref{1}, so
\begin{align*}
\lim_{r\to 0^+}&\frac{r^{N+\a}}{-r^{N-1}u_p'(r)}=\lim_{r\to 0^+}\frac{(N+\a)r^{N-1+\a}}{r^{N-1+\a}\up^p(r)}=\frac{(N+\a)}{\up^p(0)} .
\end{align*}
At the interior of the segment line $(0,1)$, by computation we have
\begin{align*}
g'(r)=& \remove{\frac 1{(N+\a)(w_p)^2}\left((1+\a)r^{\a}\up^p w_p-pr^{1+\a}\up^{p-1}(w_p)^2-r^{1+\a}\up^p (-\up'')\right)\nonumber\\
=& \frac 1{(N+\a)(w_p)^2}\left((1+\a)r^{\a}\up^p w_p-pr^{1+\a}\up^{p-1}(w_p)^2-r^{1+\a}\up^p \left(\frac{N-1}r \up'+r^{\a}\up^p\right)\right)\nonumber\\
=& \frac 1{(N+\a)(w_p)^2}\left((1+\a)r^{\a}\up^p w_p-pr^{1+\a}\up^{p-1}(w_p)^2+(N-1)r^{\a}\up^pw_p - r^{1+2\a}\up^{2p}\right) \nonumber\\
=& \frac 1{(N+\a)(w_p)^2}\left((N+\a)r^{\a}\up^p w_p-pr^{1+\a}\up^{p-1}(w_p)^2 - r^{1+2\a}\up^{2p}\right) \nonumber\\
=&\frac{(N+\a)}r g(r)-p\frac{w_p}{\up}g(r)-\frac{(N+\a)}r (g(r))^2\nonumber\\
=&} -\frac{(N+\a)}r g(r)\left(g(r)-1+p\frac{rw_p}{(N+\a)\up}\right).%\label{g'}
\end{align*}
Hence, in any possible critical point $\hat r$ we have $g(\hat r)=1-p\frac{\hat rw_p(\hat r)}{(N+\a)\up(\hat r)}<1$. This implies that $g(r)$ achieves its global strict maximum at  $r=0$ and completes the proof of Step 2. \\
%\end{proof}

\paragraph{Step 3 - $k\ge 2$.}
We check that for all $p$ and $k\ge 2$, we have $\L_{i,k}(p)>1$ for any $i\geq 1$.
Again, it suffices to analyze the first eigenvalue $\L_{1,k}(p)$. By the classical Rayleigh-Ritz
variational characterization of the first eigenvalue we have that
\[ \L_{1,k}=\inf_{v\in H^1_{0,rad}(\B )\, v\neq 0}\frac{\int_0^1 r^{N-1}(v')^2 \, dr+\mu_k \int_0^1 r^{N-3}v^2\, dr}{p\int_0^1 r^{N-1+\a}\up^{p-1}v^2\, dr}.
\]
This easily gives that  $\L_{1,k}>\L_{1,2}>1$ for any $k> 2$, and implies in turn that $\L_{i,k}>1$ for any $i\geq 1$ if $k\geq 2$. \\

\paragraph{Step 4 - $k=1$.}
We eventually show that  $\L_{2,1}(p)>1$ for all $p$. If $\L_{1,1}(p)>1$ there is nothing left to prove.
Otherwise, if $\L_{1,1}(p)\le 1$, we take advantage from the Courant Nodal Theorem.
To this aim we  study the problem \eqref{Arad} in the space $X$ of the functions which are invariant with respect to the orthogonal group in $R^{N-1}$, i.e.
\begin{equation}\label{X}
\begin{array}{r}
X:=\big\{v\in C^{1,\g}_0(\overline \B ) \, : \, v(x_1,\dots,x_N)=v(g(x_1,\dots, x_{N-1}),x_N) \\[.1cm]
\hbox{ for any } g\in O(N-1) \big\}
\end{array}\end{equation}
By a result of Smoller and Wasserman, see \cite{SW},  the eigenspace of $-\Delta_{S^{N-1}}$ related to $\mu_k$, in $X$, is one dimensional for
any $k$.  In this way the first eigenvalue $\L_{1,1}(p)$ of \eqref{Arad} related to $k=1$ gives the second eigenvalue $\L_2$ of \eqref{A}, and the
corresponding eigenspace is one-dimensional in $X$.
Next we look at the third  eigenvalue $\L_3$ of (\ref{A}), and investigate to which eigenvalue of (\ref{Arad}) is related to.
It cannot be the second eigenvalue $\L_{2,1}(p)$ corresponding to $k=1$, because the corresponding eigenfunction $\psi_{2,1}(|x|)Y_1(\theta)$ has four nodal domains, and this contradicts the Courant's Nodal Theorem.
So it has to be related either with $\L_{2,0}(p)$ or with $\L_{1,2}(p)$.
Since both $\L_{2,0}(p)$ and  $\L_{1,2}(p)$ are  strictly greater than $1$ (by Step 1 and 2, respectively), we end up with  $\L_3(p)>1$, and eventually $\L_{2,1}(p)>\L_3(p)>1$. \\

\paragraph{Step 5 - Morse index.}
At last we address to the Morse index of $\up$. Two items may happen.
If $\L_{1,1}(p)\ge 1$, then only the first eigenvalue of \eqref{Arad} is nonnegative. As it is related to the first radial mode, its eigenspace has dimension 1 and therefore the Morse index is 1.
Otherwise, $\L_{1,1}(p) <1$, then also the second eigenvalue of \eqref{Arad} is negative, and its multiplicity is equal to $\mu_2=N$ as explained in Remark \ref{passaggioa1dim}. Since there can not be other   negative eigenvalues, we conclude that the Morse index of $\up$ is $N+1$.
\end{proof}

\begin{remark}%\marginlabel{togliere?}
The restrictive assumption $\alpha\in(0,1]$ is  needed in  Step 2, in order to prove that $\L_{1,2}>1$. On the other hand, also the arguments of the following steps make use of that inequality. Therefore, removing the assumption $\alpha\in(0,1]$ could, in principle, give rise to a huge increase of the Morse index, caused by eigenvalues of type $\Lambda_{1,k}$ with $k\ge 2$ and/or of type $\Lambda_{i,2}$ with $i\ge 2$, see \cite{GGN13nonradial}.\\
On the other hand, Cowan in \cite{cowan13supercritical} studies the degeneracy of the radial solution without any assumption on $\alpha$ but he does not investigate the Morse index of $u_p$.
%Actually \textcolor{red}{Inserire referenza sull'aumento degli autovalori}. 
%Therefore $\alpha$ has to be bounded from above, even though the bound $\alpha\le 1$ is not sharp.
\end{remark}

\section{Asymptotic behavior}\label{s3}
In this section we study the behavior of the radial solution $\up$ and of its Morse index, when $p$ is at the ends of the existence range $(1, p_{\a})$.
The study of $p$ close to $1$ is,  to our knowledge, completely new and gives as a byproduct a uniqueness result, stated in Proposition \ref{uniq}.
The case $p$ close to $p_{\a}$  is  studied also in \cite{CP} from a different point of view.

\subsection{Asymptotic behavior as $p$ goes to $1$.}

When the exponent $p$ is close to $1$, it is possible to extend to the Henon problem  the uniqueness result for
\[
\left\{
\begin{array}{ll}
-\Delta u=u^p& \hbox{ in }\Omega,\\
u>0& \hbox{ in }\Omega,\\
u=0 & \hbox{ on }\pa \Omega.
\end{array}\right.\]
This result was first proved by Lin in \cite{Lin} for the least energy solution, and then generalized by Dancer, in \cite{D1}, without any assumption on the energy. See also \cite{Gr} for some related results. The main ingredient of both proofs is the asymptotic behavior of any solution of the problem as $p$ goes to $1$.
By introducing a suitable rescaling, their proof can be adapted  to the Henon problem \eqref{1}, and the following result, which holds for any value of $\a>0$, is obtained.
\begin{theorem}\label{uniq}
Let $\a>0$ be fixed. There exists $\d=\d(\a)>0$ such that, for each $p\in(1,1+\d)$, equation \eqref{1} has a unique solution, which is radial and nondegenerate.
Moreover its Morse index is equal to 1.
\end{theorem}
Before entering the details of the proof, we introduce the notation
\[%\begin{equation}\label{first-weight}
\l_1:=\inf_{\substack{v\in H^1_{0}(\B )\\ v\neq 0}} \frac{\int_{\B }|\na v|^2\, dx}{\int_{\B }|x|^{\a}v^2\, dx}
\]%\end{equation}
for the first eigenvalue with weight $|x|^{\a}$ in the ball $\B $.
It is standard to see that $\l_1$ is attained, that the first eigenfunction is simple, has fixed sign and solves
\[%\begin{equation}\label{prima-autof-weight}
\left\{
\begin{array}{ll}
-\Delta \phi_1=\l_1|x|^{\a} \phi_1 & \text{ in }\B \\
\phi_1>0  & \text{ in }\B \\
\phi_1=0 & \text{ on }\de \B .
\end{array}\right.
\]%\end{equation}
For future convenience, we consider also the same problem in a ball of arbitrary radius $R>0$, and set $\l_R$ and $\phi_R$, respectively, the first eigenvalue and the first eigenfunction with weight $|x|^{\a}$ in $B_R(0)$.
It is clear that $v(x):=\phi_R(Rx)$ satisfies
\[
\left\{
\begin{array}{ll}
-\Delta v=R^{2+\a}\l_R|x|^{\a} v & \text{ in }\B \\
v>0  & \text{ in }\B \\
v=0 & \text{ on }\de \B .
\end{array}\right.
\]
Hence $v(x)$ is a first eigenfunction with weight $|x|^{\a}$ in $\B $ and $\l_1=R^{2+\a}\l_R$. This implies in turn that
\[\l_R=\frac 1{R^{2+\a}}\l_1.\]

The asymptotic behavior of $\up$ as $p\to 1$ is described by next Lemma.
\begin{lemma}\label{l-p-rad}
Let $p_n$ be a sequence such that $p_n\to 1$ as $n\to +\infty$, and let $u_n:=u_{p_n}$ be the unique radial solution of \eqref{1} related to $p_n$. Then
$\nor u_n\nor_{\infty}^{p_n-1}\to \l_1$  and $u_n / \nor u_n\nor_{\infty} \to \phi_1$ uniformly in $\B $, as $n\to +\infty$.
\end{lemma}
\begin{proof}%[Proof \textcolor{red}{da rivedere}]
The function $\bar u_n= u_n / \nor u_n\nor_{\infty} $ satisfies
\[
\left\{
\begin{array}{ll}
-\Delta \bar u_n=\nor u_n\nor_{\infty}^{p_n-1}|x|^{\a} \bar u_n^{p_n} & \text{ in }\B \\
\bar u_n>0  & \text{ in }\B \\
\bar u_n=0 & \text{ on }\de \B 
\end{array}\right.
\]
and $\bar u_n(0)=1$ since $u_n$ is a radial function and achieves its maximum in the origin. This implies that  $\nor u_n\nor_{\infty}^{p_n-1} $ can not vanish, because
\[\bar u_n=(-\Delta )^{-1}\big(\nor u_n\nor_{\infty}^{p_n-1}|x|^{\a} \bar u_n^{p_n}\big)\leq (-\Delta )^{-1}\big(\nor u_n\nor_{\infty}^{p_n-1} \big)\leq \nor u_n\nor_{\infty}^{p_n-1}(-\Delta )^{-1}(1).\] Besides, $\nor u_n\nor_{\infty}^{p_n-1}$ can not blow up either.
Suppose by contradiction that, up to a subsequence, $\nor u_n\nor_{\infty}^{p_n-1}\to +\infty$ as $n\to +\infty$, and take
 \[ r_n=\nor u_n\nor_{\infty}^{\frac{p_n-1}{2+\a}} , \qquad \tu_n(x):=\frac 1{\nor u_n\nor_{\infty}}u_n\left( \frac x{r_n}\right) \, \text{ for } \, x\in B_{r_n}(0).\] Then $\tu _n$ solves
\begin{equation}\nonumber
\left\{
\begin{array}{ll}
-\Delta \tu_n=|x|^{\a} \tu_n^{p_n} , & \text{ in }B_{r_n}(0), \\
\tu_n>0  , & \text{ in }B_{r_n}(0) ,
%\\ \tu_n(0)=1
\end{array}\right.
\end{equation}
and $B_{r_n}(0)$ is an expanding ball.
Moreover in each compact set $K\subset \R^N$, $|x|^{\a }\tu_n^{p_n}$ is uniformly bounded so that $\tu_n\to \tu$ uniformly on compact sets of $\R^N$, and $\tu$ is a solution of
\[
\left\{
\begin{array}{ll}
-\Delta \tu=|x|^{\a} \tu , & \text{ in }\R^N ,\\
\tu>0  , & \text{ in } \R^N .
%\\ \tu(0)=1.
\end{array}\right.
\]
Let now $\l_R$ and $\phi_R$ be respectively the first eigenvalue and the first eigenfunction with weight $|x|^{\a}$ in $B_R(0)$. If $R$ is large, we have $\l_R<1$, then
\[
0>\int_{\de B_R(0)}\tu \frac{\de \phi_R}{\de \nu} \, d\sigma=\int_{B_R(0)}\tu \Delta \phi_R-\phi_R \Delta \tu \, dx=(1-\l_R)\int_{B_R(0)}|x|^{\a}\tu \phi_R \, dx>0
\]
getting a contradiction.\\
Therefore, up to a subsequence, $\nor u_n\nor_{\infty}^{p_n-1}$ converges to some positive number $\l$, and $\bar u_n$ converges  uniformly in $\B $ to a function $\bar u$ which solves
\[
\left\{
\begin{array}{ll}
-\Delta \bar u=\l|x|^{\a} \bar u , & \text{ in }\B , \\
\bar u\geq 0 , & \text{ in }\B , \\
\bar u=0 , & \text{ on }\de \B ,
\end{array}\right.
\]
and $\bar u(0)=1$.
Then $\l$ and $\bar u$ must be respectively the first eigenvalue and the first eigenfunction with weight $|x|^{\a}$ in $\B $.
\end{proof}

Any other solution,  possibly non radial, follows the same behavior described in Lemma \ref{l-p-rad}.
\begin{lemma}\label{l-p-gen}
Let $p_n$ be a sequence such that $p_n\to 1$ as $n\to +\infty$ and let $v_n:=v_{p_n}$ be a solution of \eqref{1} related to the exponent $p_n$. Then
$\nor v_n\nor_{\infty}^{p_n-1}\to \l_1$ and $v_n / \nor v_n\nor_{\infty} \to \phi_1$ uniformly in $\B $,
as $n\to +\infty$.
\end{lemma}
\begin{proof}%[Proof \textcolor{red}{da rivedere}]
If we prove that $\nor  v_n \nor_{\infty}^{p_n-1}$ is bounded, then the thesis follows as in Lemma \ref{l-p-rad}.
Suppose by contradiction that $\nor v_n \nor_{\infty}^{p_n-1}\to +\infty$, up to a subsequence. We denote by $Q_n\in \B $ the points such that $v_n(Q_n)=\nor v_n\nor_{\infty}$. Up to a subsequence, $Q_n$ converges to some point $Q_0\in \overline{B}_1(0)$. We can distinguish some different cases.
%\\[4pt]

\noindent{\em Case 1:}  $Q_0\in \B \setminus{0}$. We set $\mu_n^2:=\nor v_n\nor_{\infty}^{p_n-1}|Q_0|^{\a}$, so that $\mu_n\to +\infty$ as $n\to +\infty$, and introduce the functions  $\tv_n(x) =\nor v_n\nor^{-1} v_n\big( \frac x{\mu_n}+Q_n\big)$,  which are defined in $\widetilde{B}_n=\{x\in \R^N, \, : \, \frac x{\mu_n}+Q_n\in \B \}$. We have
$$\left\{ \begin{array}{ll}
-\Delta \tv _n=\frac{\big|\frac x{\mu_n}+Q_n\big|^{\a}}{|Q_0|^{\a}}\tv_n^{p_n} & \text{ in } \widetilde{B}_n\\
\tv_n(0)=1\\
\tv_n=0  &\text{ on } \de \widetilde{B}_n.
\end{array}\right.$$
Notice that the sets $\widetilde{B}_n$ cover all $\R^N$ as $n\to+\infty$, and the right-hand-side of the equation is locally uniformly bounded.
Thus the sequence $\tv_n$ converge locally uniformly (up to a subsequence) to an entire nonnegative, non-null solution of $ -\Delta \tv =\tv$, and this is not possible.

\noindent{\em Case 2:} $Q_0\in \de \B $. Let $\tv_n$ and $\widetilde{B}_n$ be as in the previous case. Then, following the proof of \cite[Theorem 1.1]{GS}, it is standard to see that either $\widetilde{B}_n\to \R^N$ and $\tv_n$ converges to the same function  $\tv$ introduced in the previous case, or $  \widetilde{B}_n$ tends to the half-space $\Sigma:=\{x\in \R^N,\,\, \text{ such that }x_N>-1\}$, and $\tv_n$ converges uniformly on compact sets of $\Sigma$ to a function $\tilde w$ that solves
\[\left\{\begin{array}{ll}
-\Delta \tilde w = \tilde w \quad & \text{ in }\Sigma,\\ \tilde w\geq 0, & \text{ in }\Sigma, \\ \tilde w=0 \quad &\text{ on } \de \Sigma,\end{array}\right.\]
with $\nor \tilde w \nor_{\infty}=\tilde w(0)=1$. The first occurrence has been ruled out  in the previous case. The second one is not possible either, because a positive solution of the previous equation should be strictly increasing w.r.t.~the $x_N$ variable, contradicting the fact that the maximum is achieved in the origin, see \cite[Theorem 2]{D2}.

\noindent{\em Case 3:} $Q_0=0$ and $\nor v_n\nor_{\infty}^{\frac{p_n-1}{2+\a}}|Q_n|$ is bounded.
In this case we let $\mu_n^{2+\a}=\nor v_n\nor_{\infty}^{p_n-1}$. We may assume without loss of generality that $\mu_n Q_n$ converges to some point $\tilde Q$ in $\R^N$. We define  $\tv_n(x)=\nor v_n\nor_{\infty}^{-1}v_n\big( (x-\tilde Q)/\mu_n+Q_n\big)$ for all $x\in B_{\mu_n}(\tilde Q-\mu_nQ_n)$. We have
$$\left\{ \begin{array}{ll}
-\Delta \tv _n=\big|x-\tilde Q+\mu_nQ_n\big|^{\a}\tv_n^{p_n} & \text{ in } B_{\mu_n}(\tilde Q-\mu_nQ_n)\\
\tv_n(\tilde Q)=1\\
\tv_n=0  &\text{ on } \de B_{\mu_n}(\tilde Q-\mu_nQ_n) .
\end{array}\right.$$
It is standard to see that, up to a subsequence, $\tv_n\to \tv$ uniformly on compact sets of $\R^N$, where $\tv$ is an entire,  positive and bounded solution to
$-\Delta \tv =|x|^{\a}\tv$. As explained in the proof of Lemma \ref{l-p-rad}, this is not possible.

\noindent{\em{Case 4:}} $Q_0=0$  and $\nor v_n\nor_{\infty}^{\frac{p_n-1}{2+\a}}|Q_n|$ is unbounded.
We let $\mu_n^{2}=\nor v_n\nor_{\infty}^{p_n-1}|Q_n|^{\a}$. By hypothesis $\mu_n^2>\frac C{|Q_n|^{2+\a}}|Q_n|^{\a}=\frac C{|Q_n|^{2}}\to +\infty$ as $k\to +\infty$. The rescaled function $\tv_n=\nor v_n\nor_{\infty}^{-1}v_n\big( \frac x{\mu_n}+Q_n\big)$ satisfies in $\widetilde B_n$
$$\left\{ \begin{array}{ll}
-\Delta \tv _n=\Big(\frac{|x+\mu_nQ_n|}{\mu_n |Q_n|}\Big)^{\a}\tv_n^{p_n} & \text{ in } \widetilde{B}_n\\
\tv_n(0)=1\\
\tv_n=0  &\text{ on } \de \widetilde{B}_n.
\end{array}\right.$$
Observe that, up to a subsequence $\frac {Q_n}{|Q_n|}\to \tilde Q\in \de \B $ and $\mu_n |Q_n|= \nor v_n\nor_{\infty}^{\frac{p_n-1}2}|Q_n|^{\frac{2+ \a}2}\to +\infty$ as $k\to +\infty$ by hypothesis. Again $\widetilde B_n$ is an expanding domain and, up to a subsequence $\tv_n\to \tv$ uniformly on compact sets of $\R^N$, where $\tv$ solves
$$-\Delta \tv =\tv \quad \text{ in }\R^N,\quad \tv\geq 0, \quad \tv(0)=1$$
and this is not possible as before. This case concludes the proof of the Lemma.
\end{proof}

We are now ready to prove the main result of this subsection.

\begin{proof}[Proof of Theorem \ref{uniq}]% \textcolor{red}{da rivedere}
We argue by contradiction. Suppose there exist a sequence $p_n\to 1$ and functions $u_n$ and $v_n$ that are different solutions of \eqref{1} related to the same exponent $p_n$. First we show that the difference $u_n-v_n$ must change sign in $\B $. Using equation \eqref{1} we have
\[ 0=\int_{\B }u_n\Delta v_n-v_n\Delta u_n\, dx=\int_{\B }|x|^{\a}u_nv_n\big(u_n^{p_n-1}-v_n^{p_n-1}\big)\, dx.\]
If $u_n\geq v_n$ then $u_n\equiv v_n$ and we are done.  Let $\varphi_n:=(u_n-v_n)/\nor u_n-v_n\nor_{\infty}$. The function $\varphi_n$ satisfies
\begin{equation}\label{***}
\left\{ \begin{array}{ll}
-\Delta  \varphi_n=|x|^{\a} w_n\, \varphi_n , & \text{ in } \B ,\\
\nor \varphi_n\nor_{\infty}=1, &\\
\varphi_n=0  , &\text{ on } \de \B ,
\end{array}\right.
\end{equation}
where $\displaystyle w_n(x)=p_n\int_0^1\left(t u_n(x)+(1-t)v_n(x)\right)^{p_n-1} dt $ is contained among $p_nu_n^{p_n-1}(x)$ and  $p_nv_n^{p_n-1}(x)$. An immediate consequence of  Lemma \ref{l-p-gen} is that both $p_nu_n^{p_n-1}$ and $p_nv_n^{p_n-1}$ go to the constant function $\l_1$ locally uniformly in $\B $. Hence $ w_n\to \l_1$ also, and therefore $\varphi_n$ converges uniformly in $\B $ to a function $\varphi$, which has $\|\varphi\|_{\infty}=1$ and solves
\begin{equation}\nonumber
\left\{\begin{array}{ll}
-\Delta \varphi=|x|^{\a}\l_1 \varphi , & \quad \text{ in }\B , \\
\varphi=0 , &\quad \text{ on }\de \B .\end{array}\right.
\end{equation}
So $\varphi$ is the first eigenfunction with weight $|x|^{\a}$ and has one sign in $\B $. But this clashes with the uniform convergence of the sign-changing functions $\varphi_n$.
Indeed no nodal region of $\varphi_n$ can disappear. So assume, by contradiction, that there exists a nodal region $A_n$ of $\varphi_n$ such that $meas(A_n)\to 0$. The function $\varphi_n$ satisfies \eqref{***} and $|x|^{\a} w_n(x)\leq C$ in $\B $. Multiplying \eqref{***} by $\varphi_n$ and integrating over $A_n$, we get, by the Poincar\'e Inequality
$$\int_{A_n}|\na \varphi_n|^2\leq C\int_{A_n} \varphi_n^2\leq \frac{C \, \mathit{meas}(A_n)}{\omega_N}\int_{A_n}|\na \varphi_n|^2$$
where $\omega_N$ is the measure of the $N$-dimensional sphere,
and this implies that $\mathit{meas} (A_n)\geq \frac{\omega_N}C$ so that $A_n$ cannot disappear.\\
Concerning the  eigenvalues problem \eqref{A}, we will show that $\L_2(p_n)\to  {\l_i}/{\l_1}$ as $n\to +\infty$ for some $i\geq 2$. Then $\L_2(p_n)>1$ and this yields that the Morse index of $u_{p_n}$ is 1, provided that $n$ is large.\\
Let $v_{2,n}$ be a second eigenfunction of \eqref{A} related to the exponent $p_n$, with eigenvalue $
\L_2(p_n)$, normalized in the $L^{\infty}$-norm. Then
\begin{equation}\nonumber
\left\{ \begin{array}{ll}
-\Delta  v_{2,n}=\L_2(p_n)p_n|x|^{\a} u_n^{p_n-1}\, v_{2,n}  & \text{ in } \B ,\\
\nor v_{2,n}\nor_{\infty}=1, &\\
v_{2,n}=0  , &\text{ on } \de \B .
\end{array}\right.
\end{equation}
As before $p_nu_n^{p_n-1}\to \l_1$ as $n\to +\infty$, while $\L_2(p_n)\leq C$. Then, up to a subsequence, $\L_2(p_n)\to \L_2$ and  $v_{2,n}$ converges uniformly in $\B $ to a function $v_2$ which has $\nor v_2\nor_{\infty}=1$ and solves
$$\left\{ \begin{array}{ll}
-\Delta  v_{2}=\L_2\l_1 |x|^{\a} \, v_{2} , & \text{ in } \B ,\\
%\nor v_{2,n}\nor_{\infty}=1, &\\
v_{2}=0  , &\text{ on } \de \B .
\end{array}\right.$$
So $v_2$ is an eigenfunction with weight $|x|^{\a}$ related to an eigenvalue $\l_i=\L_2\l_1$. Moreover, as before we have  that $v_2$ changes sign in $\B $ so that $\l_i\geq \l_2$. This implies that $\L_2(p_n)\to  {\l_i}/{\l_1}$ and concludes the proof.

\remove{  Lemma \ref{l-p-gen} we have that $p_nu_n^{p_n-1}=\nor u_n\nor_{\infty}^{p_n-1}\bar u_n^{p_n-1}\to \l_1$ uniformly on compact sets of $\B $ and the same holds for $p_nv_n^{p_n-1}$. This implies that $v_n(x)\to \l_1$ uniformly on compact sets of $\B $. Up to a subsequence $\varphi_n\to \varphi$ uniformly on $\B $ where $\varphi$ is a solution of
$$-\Delta \varphi=|x|^{\a}\l_1 \varphi \quad \text{ in }\B ,\quad \varphi=0 \quad \text{ on }\de \B $$
and this mean that $\varphi=\phi_1$, the first eigenfunction with weight $|x|^{\a}$, and so it has one sign in $\B $. The uniform convergence of $\varphi_n$ to $\varphi$ in $\B $ contradicts the fact that $\varphi_n$ changes sign in $\B $ and proves the claim.}
%{\color{red}{Spiegare meglio perche' non puo' cambiare segno??? \underline{Infatti non mi \`e chiaro: sono gli argomenti dell'appendice di \cite{D1}?}}}
%Concerning the  eigenvalues problem \eqref{A}, it follows that  $\L_i(p_n)\to  {\l_i}/{\l_1}$ as $n\to +\infty$, for any integer $i$. Moreover
%{\color{red}{scrivere bene che gli autovalori convergono con la stessa molteplicita' e le autofunzioni convergono ad una combinazione lineare delle autofunzioni corrispondenti}}.
%In particular, $\L_2(p_n)\to \frac{\l_2}{\l_1}>1$. By Theorem \ref{t1}, this yields that the Morse index of $u_{p_n}$ is 1, provided that $n$ is large.
\end{proof}
\noindent Estimating the eigenvalues $\L_i(p_n)$ from above, it can be proved that $\L_i(p_n)\to \frac{\l_i}{\l_1}$ for any $i=1,2,\dots.$ whit the same multiplicity. But this requires some computation and goes beyond this discussion.
%\begin{remark}
%As we said in the Introduction, the uniqueness result of Theorem \ref{uniq}, for $p$ close to $1$, provides an alternative proof of the asymptotic behavior of $\alpha^*(p)$ (see \cite{SSW}) for $p \to 1$.
%\end{remark}
\begin{remark}
The proof of Theorem \ref{uniq} and of Lemma \ref{l-p-gen} can be generalized to the Henon problem in any  bounded smooth domain  $\Omega$ in $\R^N$.
Actually, the problem
\[
\left\{
\begin{array}{ll}
-\Delta u=|x|^{\a}u^p & \hbox{ in }\Omega,\\
u>0& \hbox{ in }\Omega,\\
u=0 & \hbox{ on }\pa \Omega,
\end{array}\right.
\]
has a unique nondegenerate solution of Morse index one, for every $p$ in a right neighborhood of $1$.
%{\color{red}{spiegare i cambiamenti che bisogna fare nella dimostrazione e fornire dim della nondegenerazione???  vedi Dancer}}

\end{remark}

\subsection{Asymptotic behavior as $p$ goes to $p_{\a}$.}

When $p$ approaches the critical exponent $p_{\a}$, solutions blow up in the sup-norm.
\begin{lemma}\label{p-to-palfa}
Let $p_n$ be a sequence such that $p_n\to p_\a$ as $n\to +\infty$. Let $v_n$ be any solution of \eqref{1} related to the exponent $p_n$. Then $\nor v_n\nor_{\infty}\to +\infty$ as $n\to +\infty$.
\end{lemma}
\begin{proof}
By contradiction, let us suppose that $\nor v_n\nor_{\infty}$ stays bounded, possibly up to a subsequence.  Its normalized function $\bar v_n:= v_n/\nor v_n\nor_{\infty}$ satisfies
\begin{equation}\nonumber
\left\{\begin{array}{ll}
-\Delta \bar v_n=\nor v_n\nor_{\infty}^{p_n-1}|x|^{\a}\bar v_n^{p_n} & \text{ in }\B \\
\bar v_n>0 & \text{ in }\B \\
\bar v_n=0 & \text{ on }\de \B ,
\end{array}\right.
\end{equation}
and the quantity $\nor v_n\nor_{\infty}^{p_n-1}|x|^{\a}\bar v_n^{p_n}$ is uniformly bounded in $\B $. Then $\bar v_n$ converges   uniformly in $\B $ to a function $\bar v$ which solves
\begin{equation}\nonumber
\left\{\begin{array}{ll}
-\Delta \bar v=L|x|^{\a}\bar v^{p_{\a}} & \text{ in }\B \\
\bar v>0 & \text{ in }\B \\
\bar v=0 & \text{ on }\de \B 
\end{array}\right.
\end{equation}
where $L\geq 0$ is, up to a subsequence, the limit of $\nor v_n\nor_{\infty}^{p_n-1}$ as $n\to +\infty$. Moreover $\nor \bar v\nor_{\infty}=1$. If $L>0$ then we get a contradiction with the Pohozaev identity, if, else, $L=0$ then the function $\bar v$ is harmonic in $\B $ and hence it has to be constant. The boundary conditions then implies $\bar v\equiv 0$ contradicting that $\nor \bar v\nor_{\infty}=1$.
\end{proof}

A rescaling of the $x$ variable is needed to put in evidence the character of $\up$ as $p\to p_{\a}$. In  that way, the blowup of the supnorm is changed into a blowup of the domain.

\begin{proposition}\label{asym-p-alfa}
Let $p_n$ be a sequence such that $p_n\to p_\a$ as $n\to +\infty$ and let $u_n:=u_{p_n}$ be the unique radial solution of \eqref{1} related to $p_n$.
We next set  $\mu_n:= \nor u_n\nor_{\infty}^{\frac{p_n-1}{2+\a}}$, and
\[ \widetilde{u}_n(x):=\frac 1{\nor u_n \nor_{\infty}}u_n\left( \frac x{\mu_n}\right) \quad \text{ as } x\in B_{\mu_n}(0).\]
Then, as $n\to+\infty$, the function $\widetilde u_n$ converges in $C^{\infty}_{\mathrm{loc}}(\R^N)$ to the function
\begin{equation}\label{U}
 U(x)=\frac 1{\left( 1+C_{\alpha}|x|^{2+\a}\right)^{\frac{N-2}{2+\a}}},  \qquad  C_{\alpha} = \frac{1}{(N-2)(N+\alpha)}, \end{equation}
which is the unique radial bounded solution of
\begin{equation}\label{eq-U}
\left\{\begin{array}{ll}
-\Delta U=|x|^{\a}U^{p_{\a}}\qquad &\text{ in }\R^N\\
U\geq 0 &\text{ in }\R^N\\
U(0)=1.
\end{array}\right.
\end{equation}
Besides for every $n$
\begin{equation}\label{stimautilde}
\widetilde u_n(x)\le U(x) \qquad \text{ as } x\in B_{\mu_n}(0).
\end{equation}
\end{proposition}
\begin{proof}
It is easy to check that every $\widetilde u_n$ has maximum equal to $1$ in $x=0$ and  solve
\[
\left\{\begin{array}{ll}
-\Delta \widetilde u_n=|x|^{\a}\widetilde u^{p_n} \qquad & \hbox{ in } B_{\mu_n}(0),\\
\widetilde u_n>0& \hbox{ in } B_{\mu_n}(0),\\
\widetilde u_n =0 & \hbox{ on }\pa B_{\mu_n}(0).
\end{array}\right.\]
Hence standard elliptic theory implies that $\widetilde u_n$ converges  to a radial bounded function $U$ that solves \eqref{eq-U}.
By \cite{GS81global}, the problem  \eqref{eq-U} has an unique radial bounded  solution, given by \eqref{U}.

One further transformation is useful to obtain estimate \eqref{stimautilde}:
\[ t= (N-2)^{\frac{2(N-2)}{2+\alpha}}|x|^{-(N-2)}, \qquad y_n(t)=\widetilde u_n(x) \]
for $t\ge T_n=\left({(N-2)^2}/{\nor u_n\nor_{\infty}^{p_n-1}}\right)^{\frac{N-2}{2+\alpha}}$.
The functions $y_n$ solve classical Emden-Fowler equations with the same parameter
\[ \kappa=\frac{2(N-1)+\alpha}{N-2}>2.\]  Actually every $y_n$ is characterized as the unique solution to
\begin{align*}\label{KF}
\left\{\begin{array}{l}
y''_n + t^{-\kappa}y^{p_n} = 0 , \\
\lim\limits_{t\to+\infty}y_n(t)=1.
\end{array}\right.
\end{align*}
In \cite{AP86} it has been proved that
\[ y_n(t)\le \left(1+ \dfrac{1}{(\kappa-1)t^{\kappa-2}}\right)^{-\frac{1}{\kappa-2}} ,\]
for $t\ge T_n$, which is equivalent to \eqref{stimautilde}.
\end{proof}

Eventually we are able to prove that the Morse index of $\up$ is $N+1$, for $p$ close to $p_{\alpha}$.
\begin{proposition}\label{uniq2} Let $\alpha\in(0,1]$ fixed. There is $\delta>0$ such that, for all $p\in(p_{\alpha}-\delta,p_{\alpha})$, the radial solution $u_p$ of \eqref{1} is nondegenerate and its  Morse index  is equal to $N+1$.
\end{proposition}
\begin{proof}
By Theorem \ref{t1}, it suffices to show that $\L_{1,1}(p)<1$, for $p$ in a suitable left neighborhood of $p_{\alpha}$. To this end, we first remark that
\[ \L_{1,1}(p) = \inf\limits_{v\in H^1_{0,\mathrm{rad}}(\B )} \dfrac{\int_0^1 r^{N-1}(v')^2 dr + (N-1)\int_0^1 r^{N-3}v^2 dr}{p\int_0^1r^{N-1+\alpha}\up^{p-1}v^2 dr}.
\]
Let $\phi$ be a cut-off function ($\phi\equiv 1$ in $B_{1/3}(0)$ and $\phi\equiv 0$ outside $B_{2/3}(0)$), and take $v=-u_p' \phi$  as a test function. It gives
\begin{align*}
\L_{1,1}(p) \le \dfrac{\int_0^1 r^{N-1}u''_p\left(u'_p\phi^2\right)' dr + \int_0^1 r^{N-1}(u'_p\phi')^2 dr + (N-1)\int_0^1 r^{N-3}(u_p'\phi)^2 dr}{p\int_0^1r^{N-1+\alpha}\up^{p-1}(u_p'\phi)^2 dr}.
\end{align*}
Integrating by parts the first integral in the right hand side yields
\begin{align*}
\int_0^1 r^{N-1}u''_p\left(u'_p\phi^2\right)' dr= -\int_0^1 \left(r^{N-1}u''_p\right)'u'_p\phi^2 dr =\\
-(N-1)\int_0^1 r^{N-3}\left(u'_p\phi\right)^2 dr+p\int_0^1 r^{N-1+\alpha} u_p^{p-1}(u'_p\phi)^2 dr
+ \alpha\int_0^1 r^{N-2+\alpha} u_p^{p}u'_p\phi^2 dr
\end{align*}
by equation \eqref{wp}. Therefore
 \begin{align*}
\L_{1,1}(p) \le 1 +\dfrac{ \int_0^1 r^{N-1}(u'_p\phi')^2 dr + \alpha\int_0^1 r^{N-2+\alpha} u_p^{p}u'_p\phi^2 dr}{p\int_0^1r^{N-1+\alpha}\up^{p-1}(u_p' \phi)^2 dr},
\end{align*}
and the thesis follows by checking that
\begin{equation}\label{dis} \int_0^1 r^{N-1}(u'_p\phi')^2 dr <-\alpha\int_0^1 r^{N-2+\alpha} u_p^{p}u'_p\phi^2 dr\end{equation}
for $p$ near $p_{\alpha}$.
We use the notations introduced in Proposition \ref{asym-p-alfa} and perform the change of variable $\rho=\mu_p r$; it gives
\begin{align*}
\int_0^1 r^{N-1}(u'_p\phi')^2 dr = & \nor u_p\nor_{\infty}^{2}\, \mu_p^{4-N}\int_0^{\mu_p} \rho^{N-1}\left( \widetilde{u}'_p \widetilde{\phi}'_p\right)^2 d\rho,
\\
\int_0^1 r^{N-2+\alpha} u_p^{p}u'_p\phi^2 dr = & \nor u_p\nor_{\infty}^{2}\, \mu_p^{4-N}\int_0^{\mu_p} \rho^{N-2+\alpha} \widetilde{u}_p^{p}\, \widetilde{u}'_p \widetilde{\phi}_p^2 d\rho ,
\end{align*}
where $\widetilde{\phi}_p(\rho)=\phi(\rho/\mu_p)$. So inequality \eqref{dis} is equivalent to
\[\int_0^{\mu_p} \rho^{N-1}\left( \widetilde{u}'_p \widetilde{\phi}'_p \right)^2 d\rho <  -\alpha\int_0^{\mu_p} \rho^{N-2+\alpha} \widetilde{u}_p^{p}\, \widetilde{u}'_p \widetilde{\phi}_p^2 d\rho ,
\]
for $p$ close to $p_{\alpha}$, and we prove it by showing that
\[\lim\limits_{p\to p_{\alpha}}\int_0^{\mu_p} \rho^{N-1}\left( \widetilde{u}'_p \widetilde{\phi}'_p \right)^2 d\rho <  \alpha \lim\limits_{p\to p_{\alpha}}\int_0^{\mu_p} \rho^{N-2+\alpha} \widetilde{u}_p^{p}\, (-\widetilde{u}'_p) \widetilde{\phi}_p^2 d\rho .
\]
Indeed the term in the left side vanishes, because \eqref{stimautilde} implies that
\begin{align*}
-\widetilde{u}'_p(\rho) =  &  \dfrac{1}{\rho^{N-1}} \int_0^{\rho} r^{N-1+\alpha} \widetilde{u}^p_p(r) dr \le \dfrac{C}{\rho^{N-1}} \int_0^{+\infty} r^{N-1+\alpha} U^p_p(r) dr \le \dfrac{C}{\rho^{N-1}}
\end{align*}
for a new constant $C$, if $(N+\alpha)/(N-2)<p<p_{\alpha}$, and therefore
\begin{align*}
\int_0^{\mu_p} \rho^{N-1}\left( \widetilde{u}'_p \widetilde{\phi}'_p \right)^2 d\rho  \le \dfrac{C}{\mu_p^2} \int_0^{\mu_p} \rho^{-(N-1)} \left(  \phi'_p \left(\dfrac{\rho}{\mu_p}\right)\right)^2 d\rho = \dfrac{C}{\mu_p^N} \int_0^{1} r^{-(N-1)} \left(  \phi'_p (r)\right)^2 dr.
\end{align*}
Concerning the right side, we have by the same estimates that
 \begin{align*}
 \rho^{N-2+\alpha} \widetilde{u}_p^{p} (-\widetilde{u}'_p) \widetilde{\phi}_p^2 \le C \rho^{\alpha-1-p(N-2)} \le C \rho^{-(1+\varepsilon)}
\end{align*}
provided that $(\alpha+\varepsilon)/(N-2)\le p<p_{\alpha}$.
So we may pass to the limit inside the integral and obtain
\[\lim\limits_{p\to p_{\alpha}}\int_0^{\mu_p} \rho^{N-2+\alpha} \widetilde{u}_p^{p}\, (-\widetilde{u}'_p) \widetilde{\phi}_p^2 d\rho =
\int_0^{+\infty} \rho^{N-2+\alpha} U^{p}\, (-U'_p) d\rho >0.
\]
\end{proof}
\begin{remark}
The nondegeneracy of the radial solution can be used, for example, to find solutions of problem \eqref{1} if $\Omega$ is a suitable perturbation of $\B$ and the exponent $p$ is supercritical. This is done in \cite{cowan13supercritical} and also in \cite{GG} but from another point of view.
\end{remark}
%\subsection{Consequences}
%Now we are in position to prove the following result:
%

\section{The bifurcation result}\label{s4}

In this section we prove Theorem \ref{tbif}, namely we show that there is at least one branch of positive nonradial solutions that leads off from the curve of radial solutions, and that is unbounded in the Holder space $C^{1,\g}_0(\overline \B )$.
We shall only put in evidence the outline of the proof and give a quick sketch of the technical details, because  Theorem \ref{tbif} follows from the results obtained in the previous sections in a way similar to \cite[Theorems 2.1 and 3.3]{G} (see also \cite{AM}).

Before entering the details, we recall some notions and fix some notations.
The couple  $(\bar p,u_{\bar p})$ is said a {\em  nonradial bifurcation point } if in  every neighborhood of $(\bar{p},u_{\bar{p}})$  in the product space $(1,p_{\a})\times C^{1,\g}_0(\overline \B )$ there exists a couple $(p,v)$ such that $v$ is a nonradial solution of (\ref{1}) related to the exponent $p$.
If $(\bar{p},u_{\bar{p}})$ is a bifurcation point, then $\bar p$ must be a degeneracy point for $u_p$, i.e.
%Besides, $\bar p$ is a degeneracy point if
the related radial solution  $u_{\bar p}$ has to be degenerate. We have proved in Theorem \ref{t1} that  these degeneracy points must satisfy $\L_{1,1}(\bar{p})=1$. We also say that a degeneracy point is a {\em Morse index changing point} if, in addition, the quantity $\L_{1,1}(p)-1$ changes sign at $\bar{p}$.
%Besides, $\bar p$ is a degeneracy point if the related radial solution  $u_{\bar p}$ is degenerate. We have proved in Theorem \ref{t1} that  these degeneracy points must satisfy $\L_{1,1}(\bar{p})=1$. We also say that a degeneracy point is a {\em Morse index changing point} if, in addition, the quantity $\L_{1,1}(p)-1$ changes sign at $\bar{p}$.
It has to be noticed that degeneracy points do exist, and they are a finite number.

\begin{proposition}\label{deg-points}
For any $\a\in(0,1]$ there exists a finite number of degeneracy points.
\end{proposition}
\begin{proof}
By Theorem \ref{t1}, the degeneracy points are the zeros of the map $p \mapsto \Lambda_{1,1}(p) -1$.
Because the arguments in the proofs of Theorems \ref{uniq} and \ref{uniq2} yield that $\Lambda_{1,1}(p)-1$ changes sign in $(1,p_{\a})$,  the thesis follows once we prove that $\Lambda_{1,1}(p)$ is real analytic.
To this end it suffices to check that  $u_p$ is real analytic w.r.t.~$p$, by a general result due to Kato \cite{K}.
\\
Let $\varphi_1$ be the first positive eigenfunction of $-\Delta $ in $\B $ with Dirichlet boundary conditions.
We show that, for every $p\in(1,p_\a)$, there are two positive constants $c$ and $C$ so that
\begin{equation}\label{cC}
c\, \varphi_1 \le u_p \le C\, \varphi_1 \quad \mbox{ in the closure of } \B .
\end{equation}
Indeed, the function ${u_p}/{\varphi_1}$ is nonnegative, radial and verifies
\[ \lim_{r\to 1^-} \dfrac{u_p(r)}{\varphi_1(r)}=\dfrac{{u_p}'(1)}{{\varphi_1}'(1)}>0 \]
by the Hopf boundary Lemma. This implies that \eqref{cC} holds at least  in a neighborhood of $\de \B $, and then it has to hold  (changing eventually the constants) in the interior of $\B $.
Next, let
\[ C_{\varphi_1}=\left\{u\in C^0_0(\B )\, : \, u \hbox{ is radial and } u/{\varphi_1} \hbox{ is bounded} \right\}, \]
\[ C_{\varphi_1}^+=\left\{u\in  C_{\varphi_1}\, : \, u >0 \hbox{ in }\B  \right\}. \]
Estimate \eqref{cC} yields that  the map $(1,+\infty)\times C_{\varphi_1}^+\ni (p,u) \mapsto u^p\in C_{\varphi_1}^+$ is analytic at any  point $(p,u_p)$ via \cite[Proposition 1]{D1}.
Then also the map $F: (1,p_\a)\times C_{\varphi_1}^+\to C_{\varphi_1}^+$,  $F(p,u)=u- \left(-\Delta\right)^{-1}\left(|x|^{\a}u^p\right)$,  is real analytic near $(p,u_p)$.
Now the curve $(p,\up)$ (as $1<p<p_\a$) is the zero-level set of the function $F$, and $\partial_u F(p,\up)$ is invertible in $C_{\varphi_1}$ by Proposition \ref{lin-rad}. Hence the analytic version of the Implicit Function Theorem gives the thesis.
\end{proof}

An immediate consequence of Proposition \ref{deg-points} and Theorems \ref{uniq}, \ref{uniq2} is the following.
\begin{proposition}\label{therearemicp}
For any $\a\in(0,1]$ there exists an odd number of Morse index changing points in $(1,p_{\a})$.
\end{proposition}

Such Morse index changing points are crucial because we are able to prove that they give rise to bifurcation. Indeed, we have:
\remove{We introduce a family of operators $T(p,v):(1,p_{\a})\times X\to X$, defined by
\[ T(p,v):=\left(-\D\right)^{-1}\left(|x|^{\a} |v|^{p-1}v\right).\]
It is easily seen that  $T$  is a compact operator for every fixed $p$ and is continuous with respect to $p$. We thus use it to perturb identity and define $S(p,v):(1,p_{\a})\times X\to X$ as
\[ S(p,v):=v-T(p,v).\]
A function $v\in X$  solves \eqref{1} with exponent $p$ if and only if  $(p,v)$ is in the kernel of $S$ and $v>0$ in $\B $.
\\
We claim that, if $\bar p$ is a Morse-index changing point, then $m(p)$, changes exactly by $1$, i.e.
\begin{equation}\label{2.4}
|m(\bar{p}+\d)-m(\bar{p}-\d)|=1
\end{equation}
as $\d>0$ is small enough.
To prove this claim, we recall that the eigenspace of the Laplace-Beltrami operator on  $S^{N-1}$, spanned by the eigenfunctions corresponding to the eigenvalue  $\mu_1=N-1$ which are $O(N-1)$ invariant, is one dimensional (see Smoller and Wasserman \cite{SW}).
On the other hand the eigenspaces of \eqref{A} are generated by the product of the radial eigenfunctions $\psi_{i,k}$ for the corresponding spherical harmonics $Y_k$
(see Remark \ref{passaggioa1dim}).
In particular the eigenspace related to $\L_{1,1}$, restricted to the space $X$, is one-dimensional, and this gives \eqref{2.4}.
\\
From the change in the  Morse index in  \eqref{2.4} it is easy to obtain the bifurcation at the point $(\bar p,u_{\bar p})$ using an argument of topological degree.
}

\begin{theorem}\label{t-bif}
If $\bar p$ is a Morse index changing point, then $(\bar p,u_{\bar p})$ is a nonradial bifurcation point.
\end{theorem}
\begin{proof}
To prove the assertion  we argue in the set $X$ introduced in \eqref{X}, i.e.~the subspace of the functions of $C^{1,\g}_0(\overline \B )$ which are invariant w.r.t.~the orthogonal group in  $\R^{N-1}$, and we denote by $m(p)$ the Morse index of $u_p$ restricted to $X$.
We claim that, if $\bar p$ is a Morse-index changing point, then $m(p)$, changes exactly by $1$, i.e.
\begin{equation}\label{2.4}
|m(\bar{p}+\d)-m(\bar{p}-\d)|=1
\end{equation}
as $\d>0$ is small enough.
To prove this claim, we recall that the eigenspace of the Laplace-Beltrami operator on  $S^{N-1}$, spanned by the eigenfunctions corresponding to the eigenvalue  $\mu_k$ which are $O(N-1)$ invariant, is one dimensional (see Smoller and Wasserman \cite{SW}).
On the other hand the eigenspaces of \eqref{A} are generated by the product of the radial eigenfunctions $\psi_{i,k}$ for the corresponding spherical harmonics $Y_k$
(see Remark \ref{passaggioa1dim}).
In particular the eigenspace related to $\L_{1,1}$, restricted to the space $X$, is one-dimensional, and this gives \eqref{2.4}.
\\
We define a family of operators $S:(1,p_{\a})\times X\to X$ as
\[ S(p,v):=v-\left(-\D\right)^{-1}\left(|x|^{\a} |v|^{p-1}v\right) .\]
$S(p,v)$ is a compact perturbation of the identity for any $p$ fixed, and it is continous with respect to $p$.
A function $v\in X$  solves \eqref{1} with exponent $p$ if and only if  $(p,v)$ is in the kernel of $S$ (and $v>0$ in $\B $).
\\
From the change in the  Morse index in  \eqref{2.4} it is easy to obtain the bifurcation at the point $(\bar p,u_{\bar p})$ using an argument of topological degree applied at the operator $S(p,v)$ in a neigborhood of $(\bar p,u_{\bar p})$ as in  \cite[Theorem 2.1]{G}, and observing  that these bifurcating solutions  are nonradial since $u_p$ is radially nondegenerate for any $p$ by Lemma \ref{lin-rad}.
\end{proof}

Theorems \ref{therearemicp} and \ref{t-bif} state the existence of at least one bifurcation of non-radial solutions from the curve of radial solutions.
There actually is a branch of nonradial solutions. To be more precise, we set $\Sigma$ the closure of the set
\begin{equation}\label{3.1}
\{(p,v)\in (1,p_{\a})\times X \, : \,  S(p,v)=0 \, ,\, v\neq \up\}
\end{equation}
where $S(p,v)$ and $X$ are as defined in the proof of Theorem \ref{t-bif}.
If $(\bar p,u_{\bar p})$ is a nonradial bifurcation point, then $(\bar{p},u_{\bar{p}})\in \Sigma$. We denote by $\cC(\bar{p})$ the closed  connected component of $\Sigma$  which contains $(\bar{p},u_{\bar{p}})$.
Arguing as in \cite[Theorem 3.3, Step 1]{G}, one shows that  $\cC(\bar{p})$ is a  branch of nonradial solutions spreading from $\bar p$.

\begin{proposition}\label{positivity}
Let $\bar{p}$ be a Morse index changing point. If $(p,v)\in \cC(\bar{p})$ then $v$ is a solution of \eqref{1} with exponent $p$. In particular $v>0$ in $\B $.
\end{proposition}

The bifurcation  is indeed global and obeys at the so called Rabinowitz alternative.
\begin{theorem}\label{t-bif-globale}
Let $\bar{p}$ be a Morse index changing point, and $\cC(\bar{p})$  as before.
Then either
\begin{itemize}\item[a)] $\cC(\bar{p})$ is unbounded in $(1,p_{\a})\times
  X$,
\end{itemize}
or
\begin{itemize}
\item[b)] there exists another Morse index changing point $q\neq \bar p$, such that $(q,u_{q})\in \cC(\bar{p})$.
\end{itemize}
\end{theorem}
\begin{proof}
Let us suppose that $\mathcal C (\bar{p})$ is bounded.
Then Proposition  \ref{uniq}, Lemma \ref{p-to-palfa} and Proposition \ref{positivity} imply that $\mathcal C(\bar{p})\subset [1+\d,p_{\alpha}-\d]\times X$ for some
  $\d>0$.
The rest of the proof follows exactly as in  \cite[Theorem 3.3, Steps 2--5]{G} and we do not report it.
\end{proof}
Theorem \ref{t-bif-globale}  shows that the branches that bifurcate from the Morse-index changing  points  are global.
For our purposes, it remains to show that at least one of them is not bounded in the space $X$.
To do this we need the following result:
\begin{proposition}\label{p1}
Let $\bar{p}$ be a Morse index changing point, and $ \mathcal{C}(\bar{p})$ as before.
If $ \mathcal{C}(\bar{p})$ is bounded, then
the number of the Morse index changing points in $ \mathcal{C}(\bar{p})$ including $(\bar{p},u_{\bar{p}})$ is even.
\end{proposition}
This result is based on an improved version of the Rabinowitz alternative due to Ize (see \cite{N}) and  uses again the Leray-Schauder degree
theory.
\begin{proof}
If
$\mathcal C(\bar{p})$ is bounded
then $b)$ of Theorem \ref{t-bif-globale} holds and $\mathcal C(\bar{p})$ must meet the curve of radial solutions, that we call $\mathcal  S$, in
at least one point $(p_i,u_{p_i})$, such that $p_i$ is a degeneracy point. But it can meet the curve $\mathcal  S$ also in other bifurcation points.
Recalling that the bifurcation points have to be related to degeneracy points $p_j$, Proposition \ref{deg-points} implies that  $\mathcal C(\bar{p})$
can meet $\mathcal  S$ at most in
finitely many bifurcation
points $(p_j,u_{p_j})$, $j=1,\dots,m$ with  $p_1<p_2<\dots<p_m$.
By the same arguments of \cite[Theorem 3.3, Steps 3 and 5]{G},
there is a bounded open set
$\mathcal O\subset (1,p_{\a})\times X$ such that $\mathcal C(\bar{p})\subset \mathcal O$ and $\de \mathcal O
\cap \Sigma=\emptyset$ with $\Sigma$ as in \eqref{3.1}. Moreover we can  assume that $\mathcal O$ does not contain points $(p,\up)$ if
$|p-p_j|\geq \e_0$ for $j=1,\dots, m$ and $\e_0>0$ such that there are not degeneracy points in
$\cup_{j=1}^m (p_j-2\e_0,p_j+2\e_0)$, again from Proposition \ref{deg-points}. \\
For $\mathcal O$ as above and $r>0$, consider the map
$$\begin{array}{llll}
S_r(p,v):&\hbox{ }\overline{ \mathcal O} &\rightarrow &\hbox{  } X\times \R\\
&(p,v)&\mapsto &\left(S(p,v)
%&(p,v)&\mapsto &\left(v-\left(-\D\right)^{-1}\left(|x|^{\a} |v|^{p-1}v\right)
,\nor v-\up \nor_X^2 -r^2\right)
\end{array}$$
where $\nor \cdot\nor_X$ stands for  the usual norm in the space $C^{1,\gamma}_0(\B )$.
Now, $\mathit{deg}\left( S_r(p,v),\mathcal O,(0,0)\right)$ is defined since
on $\de \mathcal O$ there are no  solutions of $S(p,v)=0$ different from the
radial solution $u_p$, and hence
$0=\nor v-\up \nor_X<r$ for such any solution. Furthermore the degree is
independent of $r>0$. For large $r$, $S_r(p,v)=(0,0)$ has no solutions in
$\mathcal O$, and hence has degree zero. On the other hand, for small  $r$, if
$(p,v)$ is a solution of $S_r(p,v)=(0,0)$, then $\nor v-\up\nor_X =r$, and hence
$p$ is close to one of the $p_j$, $j=1,\dots,m$. But then the sum of
local degrees of $S_r$ in the neighborhoods of each of the $p_j$ is
equal to zero, so that
\begin{equation}\label{2.sum}
0=\sum_{j=1}^m \mathit{deg}\left( S_r(p,v),\mathcal O\cap
B_{r}(p_j,u_{p_j}),(0,0)\right).
\end{equation}
In particular we choose $r<\e_0$ for $\e_0$ defined as before.
In order to compute the degree of $S_r(p,v)$ in $\mathcal O\cap B_{r}(p_j,u_{p_j})$ we
use again the homotopy invariance of the degree. Let us define
$$S_r^t(p,v)=\left(S(p,v), t(\nor v-\up\nor_X^2
-r^2)+(1-t)(2p_jp-p^2-p_j^2+r^2)\right)$$
%\e_i^2)\right).$$
for $t\in [0,1]$.
As before $\mathit{deg}\left( S_r^t(p,v), \mathcal O\cap B_{r}(p_j,u_{p_j}),
  (0,0)\right) $ is well defined  since there are no solutions on the
  boundary if $r$ is small (recall that $u_{p_j\pm r}$ are isolated if
  $r<\e_0$). Moreover the degree is independent of
  $t$. For $t=1$ we have $S_r^1(v,p)=S_r(p,v)$, while for $t=0$,
$S_r^0(p,v)=\left(S(p,v),2p_jp-p^2-p_j^2+r^2\right)$ and
\begin{eqnarray}
&& \mathit{deg}\left( S_r^0(p,v),\mathcal O\cap B_{r}(p_j,u_{p_j}),
  (0,0)\right)\nonumber\\
&&= \mathit{deg} \left(
  S(p,v),\mathcal O\cap B_{r}(p_j,u_{p_j}) ,0\right)\cdot \mathit{deg}\left( 2p_jp-p^2-p_j^2+r^2,
  \{|p-p_j|<r\},0\right).\nonumber
\end{eqnarray}
Now
$$\mathit{deg}\left( 2p_jp-p^2-p_j^2+r^2,
  \{|p-p_j|<r\},0\right)=1$$
for $p=p_j-r$ while
$$\mathit{deg}\left( 2p_jp-p^2-p_j^2+r^2,
  \{|p-p_j|<r\},0\right)=-1$$
for $p=p_j+r$. This implies that
\begin{eqnarray}
&&\mathit{deg}\left( S_r(p,v),\mathcal O\cap
B_{r}(p_j,u_{p_j}),(0,0)\right)=\nonumber\\
&&\mathit{deg}\left(
S(p_j-r,\cdot),\mathcal O_{p_j-r},0\right) - \mathit{deg}\left(
S(p_j+r,\cdot),\mathcal O_{p_j+r},0\right)\nonumber\\
&&=(-1)^{m(p_j-r)}-(-1)^{m(p_j+r)}\nonumber
\end{eqnarray}
where  we denote by $\mathcal O_p$ the set $\{v\in \mathcal O\, :\,
(p,v)\in \mathcal O\}$.\\
We conclude that if $(p_j, u_{p_j})$ is a Morse index changing
point then
$$\mathit{deg}\left( S_r(p,v),\mathcal O\cap
B_{r}(p_j,u_{p_j}),(0,0)\right)=\pm 2$$ while if $(p_j, u_{p_j})$ is
not a Morse index changing point then
$$\mathit{deg}\left( S_r(p,v),\mathcal O\cap
B_{r}(p_j,u_{p_j}),(0,0)\right)=0.$$ Since the nonzero terms in
(\ref{2.sum}) correspond  only to the Morse index changing points, and since these
terms add up to zero, there must be an even number of Morse index
changing points.
\end{proof}

Summing up we get

\begin{proof}[Proof of Theorem \ref{tbif}.]
Proposition \ref{therearemicp} states that there exists an odd number of Morse index changing points.
Such points give rise to bifurcation by  Theorem \ref{t-bif}.
Besides if some bifurcating branch $\mathcal{C}(\bar{p})$ is bounded, then it contains an even number of  Morse index changing points by Proposition \ref{p1}.
This implies, in turn, that at least one of the Morse index changing points gives rise to an unbounded branch of nonradial solutions.
\end{proof}

\section{Appendix}
We prove here some facts that have been used in Section \ref{s2}.
First  we show the equivalence between the Morse index of the radial solution $\up$ and the number of eigenvalues of \eqref{A} less than $1$.
This is a standard result and we report it only for reader's convenience.

\begin{lemma}\label{equiv-Morse-index}
The Morse index of $\up$ coincides with the number of the eigenvalues of \eqref{A} less than $1$, counted with their multiplicity.
\end{lemma}
\begin{proof}
Let $M(p)=j$ be the Morse index of $\up$ and $\widetilde{M}(p)=\tilde j$ be the number of the eigenvalues of \eqref{A} less than $1$, counted with their multiplicity. By definition there exist $j$ eigenfunctions $v_1,\dots,v_j\in H^1_0(\B )$ and $j$ eigenvalues $\l_1,\dots,\l_j$ such that $-\Delta v_n-p|x|^{\a}\up^{p-1}v_n=\l_nv_n$ in $\B $ and $v_n=0$ on $\de \B $ and $\l_n<0$ for any $n=1,\dots,j$, and $\l_{j+1}\geq 0$. For any $v\in {\mathrm{Span}} <v_1, \dots,v_j>$ then we have
$$\int_{\B }|\na v|^2 -p|x|^{\a}\up^{p-1}v^2\, dx\leq \l_j \int_{\B }v^2\, dx<0$$
so that
$$\frac{\int_{\B }|\na v|^2\, dx}{p\int_{\B }|x|^{\a}\up^{p-1}v^2\, dx}<1$$
and this implies in turn that $\tilde{j}\geq j$. Suppose by contradiction that $\tilde j>j$. Then there exists at least $j+1$ functions $\tilde v_1,\dots,\tilde v_{j+1}\in H^1_0(\B )$ such that
$$\frac{\int_{\B }|\na v|^2\, dx}{p\int_{\B }|x|^{\a}\up^{p-1}v^2\, dx}<1$$
for any $v\in {\mathrm{Span}} <\tilde v_1,\dots,\tilde v_{j+1}>$ and this implies that \[\int_{\B }|\na v|^2 -p|x|^{\a}\up^{p-1}v^2\, dx<0\] for any $v\in{\mathrm{Span}} <\tilde v_1,\dots,\tilde v_{j+1}>$, so that
$$\l_{j+1}\leq \max_{\substack{v\in<\tilde v_1,\dots,\tilde v_{j+1}> \\v\neq 0}}\frac{\int_{\B }|\na v|^2 -p|x|^{\a}\up^{p-1}v^2\, dx}{\int_{\B }v^2\, dx}<0$$
contradicting the definition of Morse index.
\end{proof}

Next we show an useful estimate for the function $u_p$.
\begin{proof}[Proof of Lemma \ref{der-seconda}]
Let $\tilde{u}_p$ be a radial minimizer for the functional
$$I[v]:=\frac{\int_{\B }|\na v|^2\, dx}{\left(\int_{\B }|x|^{\a}|v|^{p+1}\,dx\right)^{\frac 2{p+1}}}$$
in the space $H^1_0(\B )$. We can assume $\utp\geq 0$ in $\B $ otherwise we can consider $|\utp|$ instead of $\utp$. Then the function $\utp$ minimizes the functional
\begin{equation}\label{rad-funct}
Q[v]:=\frac{\int_0^1r^{N-1}( v')^2\, dr}{\left(\int_0^1r^{N-1+\a}|v|^{p+1}\,dr\right)^{\frac 2{p+1}}}
\end{equation}
in the space $H^1_{0,rad}(\B )$. This implies that $Q'_{[\utp]}(v)=0$ and $Q''_{[\utp]}(v,v)\geq 0$  for any $v\in H^1_{0,rad}(\B )$.
By computation
\begin{align}
Q'_{[u]}(v)=&\frac 1{\left(
  \int_0^1 r^{N-1+\a} u^{p+1}\, dr\right)^{\frac 4{p+1}}}\left[2\int_0^1 r^{N-1} u' v'
  \, dr\left( \int_0^1 r^{N-1+\a} u^{p+1}\, dr\right)^{\frac 2{p+1}}\right.\nonumber\\
-&\left. \frac 2{p+1}\!\int_0^1\!\!\!\! r^{N-1} (u')^2
  \, dr \left( \int_0^1\!\! r^ {N-1+\a} u^{p+1}\, dr\right)^{\frac 2{p+1}-1}\!\!\!\!\!\!\!\!\!\!\!\!\!(p+1)\int_0^1 \!\!r^{N-1+\a} u^p v \, dr\right]\nonumber
\end{align}
and hence
\begin{equation}\label{R1}
Q'_{[\tilde u_{p}]}(v)=2\frac {\int_0^1 r^{N-1}\tilde u_{p}' v'
  \, dr -\b_{p}\int_0^1 r^{N-1+\a} \tilde u_{p}^p v \, dr}{\left(
  \int_0^1 r^{N-1+\a}\tilde u_{p}^{p+1}\, dr\right)^{\frac 2{p+1}}}
\end{equation}
where
\begin{equation}\label{R2}
\b_{p}=\frac {\int_0^1r^{N-1}\left(
  \widetilde {u}_{p}'\right)^2 \,dr}{\int
  _0^1 r^{N-1+\a} \widetilde {u}_{p}^{p+1} \,dr}=\frac {\int_{\B }|\na \widetilde {u}_{p}|^2 dx}{\int
  _{\B } r^{N-1+\a} \widetilde {u}_{p}^{p+1} \,dx}.
\end{equation}
From  $Q'_{[\tilde u_{p}]}(v)=0$ for any $v\in H^1_{0,rad}(\B )$, it follows that  $\tilde u_{p}$ is a radial solution of
\begin{equation}\label{R3}
-\Delta \widetilde {u}_{p}=|x|^\a \b_{p}\widetilde {u}_{p}^{p}\quad
\hbox{ in }\B .
\end{equation}
Then  $u_{p}=\b_{p}^{\frac 1{p-1}}\widetilde {u}_{p}$, because  the radial solution $u_{p}$ of (\ref{1}) is unique.
From (\ref{R1}) and (\ref{R2}) we have
\begin{eqnarray}
&&Q''_{[\tilde u_{p}]}(v,v)=\frac 2{\left(\int_0^1 r^{N-1+\a}\tilde
  u_{p} ^{p+1}\, dr\right)^{\frac 4{p+1}}} \left\{\left( \int_0^1
    r^{N-1+\a} \tilde u_{p}^{p+1}\, dr\right)^{\frac 2{p+1}} \left(\int
  _0^1 r^{N-1}(v')^2 \, dr \right.\right.\nonumber\\
&&-p\,\b_{p}\int_0^1
  r^{N-1+\a} \tilde u_{p}^{p-1}v^2\, dr -\int_0^1 r^{N-1+\a}  \tilde
  u_{p}^p v \, dr \nonumber\\
&&\left.\left(\frac {2\int_0^1 r^{N-1}\tilde u_{p}' v' \, dr \cdot \int_0^1
    r^{N-1+\a} \tilde u_{p}^{p+1}  \, dr- \int_0^1
    r^{N-1} \left(\tilde u_{p}'\right)^2\, dr \cdot (p+1)\int_0^1
    r^{N-1+\a} \tilde u_{p}^p v\, dr }{\left(\int_0^1
    r^{N-1+\a} \tilde u_{p}^{p+1}\,
    dr\right)^2}\right)\right)\nonumber\\
&& -\left( \int_0^1
    r^{N-1} \tilde u_{p}' v' \, dr -\b_{p}\int_0^1
    r^{N-1+\a} \tilde u_{p}^p v \, dr\right) \nonumber\\
&&\left.\frac 2{p+1}\left( \int_0^1
    r^{N-1+\a} \tilde u_{p}^{p+1}\, dr \right)^{\frac
      2{p+1}-1}(p+1)\int_0^1
    r^{N-1+\a} \tilde u_{p}^p v\,dr\right\}\nonumber.
\end{eqnarray}
Then, using that $\int_0^1 r^{N-1} \left(\widetilde u_{p}'\right)^ 2
\, dr=\b_{p}\int_0^1 r^{N-1+\a} \widetilde u_{p}^{p+1}\, dr$ we get
\begin{eqnarray}
&&Q''_{[\tilde u_{p}]}(v,v)=\frac 2{\left(\int_0^1\!\! r^{N-1+\a}\tilde
  u_{p} ^{p+1}\, dr\right)^{\frac 2{p+1}}} \left\{\int
  _0^1\!\! r^{N-1}(v')^2 \, dr-p\b_{p}\int_0^1\!\!
  r^{N-1+\a} \tilde u_{p}^{p-1}v^2\, dr\nonumber\right.\\
&& -2 \frac {\int_0^1
    r^{N-1+\a} \tilde u_{p}^p v\, dr\int_0^1 r^{N-1}\tilde u_{p}' v'
    \, dr }{\int_0^1
    r^{N-1+\a} \tilde u_{p}^{p+1}\,
    dr}+(p+1)\b_{p}\frac {\left( \int_0^1
    r^{N-1+\a} \tilde u_{p}^p v\, dr\right)^2}{\int_0^1
    r^{N-1+\a} \tilde u_{p}^{p+1}\,
    dr}\nonumber\\
&& \left.-2\frac {\int_0^1 r^{N-1}\tilde u_{p}' v'
    \, dr \int_0^1
    r^{N-1+\a} \tilde u_{p}^p v\, dr}{\int_0^1
    r^{N-1+\a} \tilde u_{p}^{p+1}\,
    dr}+ 2\b_{p}\frac {\left( \int_0^1
    r^{N-1+\a} \tilde u_{p}^p v\, dr\right)^2}{\int_0^1
    r^{N-1+\a} \tilde u_{p}^{p+1}\,
    dr}\nonumber\right\}.
\end{eqnarray}
Since $Q''_{[\tilde u_{p}]}(v,v)\geq 0$, we have
\begin{eqnarray}
&&\!\!\!\!\!\!\! \int
  _0^1 \!\!r^{N-1}(v')^2 \, dr-p\b_{p}\int_0^1\!\!
  r^{N-1+\a} \tilde u_{p}^{p-1}v^2\, dr+(p+3) \b_{p}\frac {\left( \int_0^1
\!\!    r^{N-1+\a} \tilde u_{p}^p v\, dr\right)^2}{\int_0^1
\!\!    r^{N-1+\a} \tilde u_{p}^{p+1}\,
    dr}\nonumber\\
&&\!\!\!\!\!\!\!
-4 \frac {\int_0^1
    r^{N-1+\a} \tilde u_{p}^p v\, dr }{\int_0^1
    r^{N-1+\a} \tilde u_{p}^{p+1}\,
    dr}\int_0^1 r^{N-1}\tilde u_{p}' v'
    \, dr\geq 0\nonumber
\end{eqnarray}
and hence, from (\ref{R3})
$$\!\!\int
  _0^1 \!\!r^{N-1}(v')^2 \, dr-p\b_{p}\int_0^1\!\!
  r^{N-1+\a} \tilde u_{p}^{p-1}v^2\, dr+(p-1) \b_{p}\frac {\left( \int_0^1\!\!
    r^{N-1+\a} \tilde u_{p}^p v\, dr\right)^2}{\int_0^1\!\!
    r^{N-1+\a} \tilde u_{p}^{p+1}\,
    dr}\geq 0.$$
Recalling that $u_{p}=\b_{p}^{\frac 1{p-1}}\tilde u_{p}$ we get
\begin{eqnarray}\nonumber
\!\!\!\!\!\!\!\!\int
  _0^1 r^{N-1}(v')^2 \, dr-p\int_0^1
  r^{N-1+\a} u_{p}^{p-1}v^2\, dr
+(p-1)\frac {\left( \int_0^1
    r^{N-1+\a}  u_{p}^p v\, dr\right)^2}{\int_0^1
    r^{N-1+\a}  u_{p}^{p+1}\,
    dr}\geq 0
\end{eqnarray}
for any $v\in H^1_{0,rad}(\B )$.
\end{proof}

\end{document}